\documentclass[11pt,reqno]{amsart}
\usepackage{mathrsfs}
\usepackage{bbm}
    \usepackage{bm}
\usepackage{amsfonts}
\usepackage{amssymb,amsmath,amscd,amsthm,amsfonts}
\usepackage[dvipdfm,  %pdflatex,pdftexÕâÀï¾ö¶¨ÔËÐÐÎļþµÄ·½Ê½²»Í¬
            pdfstartview=FitH,
            CJKbookmarks=true,
            bookmarksnumbered=true,
            bookmarksopen=true,
            colorlinks=true, %×¢Ê͵ô´ËÏîÔò½»²æÒýÓÃΪ²ÊÉ«±ß¿ò(½«colorlinksºÍpdfborderͬʱעÊ͵ô)
            pdfborder=001,   %×¢Ê͵ô´ËÏîÔò½»²æÒýÓÃΪ²ÊÉ«±ß¿ò
            linkcolor=green,
            anchorcolor=green,
            citecolor=green
            ]{hyperref}
\usepackage[all]{xy}

\newcommand{\bqa}{\begin{equation}}
\newcommand{\eqa}{\end{equation}}
\newcommand{\bea}{\begin{eqnarray}}
\newcommand{\eea}{\end{eqnarray}}
\newcommand{\bna}{\begin{eqnarray*}}
\newcommand{\ena}{\end{eqnarray*}}
\newcommand{\bma}{\begin{pmatrix}}
\newcommand{\ema}{\end{pmatrix}}

\newcommand{\mk}{\mathfrak}

\def\bz{{\mathbb Z}}
\def\br{{\mathbb R}}

\def\sl2z{SL(2,\bz)}
\def\psl2z{PSL(2,\bz)}

\def\gl2r{GL(2,\br)}

\def\re{{\rm Re}}
\def\im{{\rm Im}}
\def\res{{\rm Res}}

\def\C{\mathbb{C}}

\def\R{\mathbb{R}}

\def\Z{\mathbb{Z}}

\newtheorem{lemma}{Lemma}[section]
\newtheorem{thm}[lemma]{Theorem}
\newtheorem{cor}{Corollary}

\theoremstyle{definition}

\newcommand{\bit}{\begin{itemize}}
\newcommand{\eit}{\end{itemize}}

\begin{document}
\centerline{\Large\bf
Central Values of $GL(2)\times GL(3)$ Rankin-Selberg $L$-functions
}
\centerline{\Large\bf
 with Applications
\footnote{This work is supported by the Natural Science Foundation
of Shandong Province (Grant No. ZR2014AQ002) and
Innovative Research Team in University (Grant No. IRT16R43).
}}
\bigskip

\centerline{\large Qinghua Pi}
\address{School of Mathematics and Statistics, Shandong University, Weihai, Weihai 264209, China}
\email{qhpi@sdu.edu.cn}
\numberwithin{equation}{section}

\bigskip

\noindent{\bf Abstract} \,{ Let $f$ be a normalized holomorphic cusp form
for $SL_2(\mathbb{Z})$ of weight $k$ with $k\equiv0\bmod 4$.
By the Kuznetsov trace formula for $GL_3(\R)$, we
obtain the first moment of central values of $L(s,f\otimes \phi)$,
where $\phi$ varies over Hecke-Maass cusp forms for $SL_3(\Z)$.
As an application, we obtain a non-vanishing result for $L(1/2,f\otimes\phi)$
and show that such $f$
is determined by $\{L(1/2,f\otimes\phi)\}$ as $\phi$ varies.}

\medskip



\section{Introduction}
Special values of $L$-functions are expected to carry
important information on relevant arithmetic and geometric objects.
In
1997, Luo and Ramakrishnan \cite{LR1997} asked the question that to
what extent modular forms are actually characterized by their
special $L$-values. In the same paper, they considered
the moment of $\chi_d(p)L(1/2,f\otimes\chi_d)$ as $d$ varies,
and showed that a cuspidal
normalized holomorphic Hecke newform $f$ is uniquely determined by the family
$\{L(1/2,f\otimes\chi_d)\}$ for all quadratic characters $\chi_d$.
Since then, this problem has been studied by many authors
(\cite{Lu1999}, \cite{CD2005},
\cite{Li2007}, \cite{Li2009}, \cite{GHS2009},  \cite{Mu2010},
\cite{Liu2010}, \cite{Pi2010}, \cite{Liu2011}, \cite{Zh2011},
\cite{Ma2014},  \cite{Pi2014},
\cite{Su2014},\cite{MS2015}).

Let $f$ be a normalized holomorphic Hecke-cusp form for $SL_2(\Z)$ of fixed weight $k$ with $k\equiv 0\bmod 4$.
Let $\{\phi\}$ be a Hecke basis of the space of Maass cusp forms for $SL_3(\Z)$.
In this paper, we consider central values
of Rankin-Selberg $L$-functions $L(s,f\otimes \phi)$.
By calculating the twisted moment of $A_\phi(p,p)L(1/2,f\otimes \phi)$
where $A_\phi(p,p)$ is the Hecke eigenvalue of $\phi$ at $(p,p)$,
we show that
$f$ is uniquely determined
by the family $\{L(1/2,f\otimes\phi)\}$
as $\phi$ varies over a Hecke basis of the space of Maass cusp forms for $SL_3(\Z)$.
\medskip

To state our result, we give the following notations.
\bit
\item
For $\phi$ a Hecke-Maass cusp form for $SL_3(\Z)$, let $\bm\mu_\phi=(\mu_1,\mu_2,\mu_3)$
be the Langlands parameter of $\phi$. We know that $\bm\mu_\phi$ is a point in
the region
\bna
\Lambda_{1/2}':=\left\{(\mu_1,\mu_2,\mu_3)\in\C^3,\quad
\begin{aligned}
&| \re(\mu_j)|\leq \frac{1}{2},\quad  \mu_1+\mu_2+\mu_3=0,\\
&\{\mu_1,\mu_2,\mu_3\}=\{-\overline\mu_1,-\overline\mu_2,-\overline\mu_3\}
\end{aligned}
 \right\}
\ena
in the Lie algebra $\mk {sl}_3(\C)$.
Let
\bna
\nu_1=\frac{1}{3}(\mu_1-\mu_2),\quad \nu_2=\frac{1}{3}(\mu_2-\mu_3),\quad \nu_3=-\nu_1-\nu_2
\ena
which are known as the spectral coordinates.

\item Fix a point $\bm\mu^0\in\Lambda'_{1/2}$ such that
\bna
|\mu^0_j|\asymp \|\bm \mu^0\|:=T,\quad 1\leq j\leq 3.
\ena
 As in \cite{BB2015} (or see \cite{HLZ2017}),
we choose the test function $h(\bm\mu)$ to
localize at a ball of radius $M=T^\theta$  with $0<\theta<1$
about $w(\bm \mu^0)$, where $w$ are elements in the Weyl group $W$.
For a precise definition of $h(\bm\mu)$, we refer to section \ref{subsec-choice-testfunc}.\item Let $d\bm\mu=d\mu_1 d\mu_2$ and $d_{\mathrm{spec}}\bm\mu=\mathrm{spec}(\bm\mu)d\bm\mu$ with
\bna
\mathrm{spec}(\bm\mu)=\prod_{j=1}^3\left(3\nu_j\tan\left(\frac{3\pi}{2}\nu_j \right)\right).
\ena
\eit
\medskip

Our main result is in the following.
\begin{thm}\label{thm-main}
Let $f$ be a normalized holomorphic Hecke cusp form for $SL_2(\Z)$ of weight $k$ with $k\equiv 0\bmod 4$.
Let $\{\phi\}$ be a Hecke basis of the space of cusp forms for $SL_3(\Z)$
and $A_\phi(p,p)$ be the Hecke eigenvalue of $\phi$ at $(p,p)$.
One has
\bea
\sum_{\phi}\frac{h(\bm\mu_\phi)}{\mathcal N_\phi}A_\phi(p,p)L(1/2,f\otimes \phi)
=\frac{\lambda_f(p)}{p^{3/2}}M_k(h)
+O_{k,\epsilon}(p^{\frac{7}{32}+\epsilon}T^{\frac{5}{2}+\epsilon}M^2)\label{main-formula}
\eea
for $T\gg_{k,\epsilon} p^{3+\frac{7}{16}+\epsilon}$.
Here $\mathcal N_\phi$ is the normalized factor defined
in (\ref{normalized-factor-cusp}) and
\bna
M_k(h)=
\frac{1}{192\pi^5}\iint_{\re(\bm\mu)=0}h(\bm \mu)\left(1+
\prod_{j=1}^3\frac{\Gamma(\frac{k}{2}+\mu_j)}
{\Gamma(\frac{k}{2}-\mu_j)}\right)d_{\mathrm{spec}}\bm \mu.
\ena
\end{thm}
Note that
$M_k(h)\asymp_{k}T^3M^2$.
On taking $p=1$, the above theorem implies the existence
of non-vanishing  of $L(1/2,f\otimes \phi)$ as $\phi$ varies.
Moreover, by the strong multiplicity one theorem (see \cite{PS1979}),
we have the following corollary.
\begin{cor}
Let $f$ and $f'$ be two normalized holomorphic cusp forms for $SL_2(\Z)$
of fixed weight $k$ with $k\equiv0\bmod 4$.
If $L(1/2,f\otimes\phi)=L(1/2,f'\otimes\phi)$ for all Hecke-Maass cusp forms
$\phi$ for $SL_3(\Z)$,
then $f=f'$.
\end{cor}

We remark that central values of $L(s,f\otimes\phi)$
vanish for $k\equiv2\bmod 4$. In this case we can consider $\frac{d}{ds}L(1/2,f\otimes\phi)$ instead of $L(1/2,f\otimes g)$ as in
\cite{Zh2011}. But we do not address this here.
\medskip

This paper is arranged as follows. In section 2, we review
the Kuznetsov trace formula  in the version of \cite{Bu2014},
choose the test function and
give the approximate functional equation of the Rankin-Selberg
$L$-function.
Theorem \ref{thm-main} will be proved in section 3, where we apply the approximate functional equation
and the Kuznetsov trace formula, and give estimations on each terms.
The main term in (\ref{main-formula})
comes from the geometric term associated to the trivial Weyl's element,
and the error term comes from the maximal Eisenstein series in the continuous spectrum.
\section{Preliminaries}
In this section, we review the definition of automorphic forms on $SL(3,\Z)$
in \cite{Bl2013} (or  see \cite{Go2006}), the Kuznetsov's trace formula
in \cite{Bu2014} (or see \cite{BB2015}) and the approximate functional equation of Rankin-Selberg $L$-functions.

Let
\bna
\mathfrak h^3&=&\left\{z=\bma 1&x_2&x_3\\&1&x_1\\&&1\ema\bma
y_1y_2&\\&y_1\\&&1\ema,\quad x_1,x_2,x_3\in \mathbb R,y_1,y_2\in
\mathbb R^+\right\}\\
&\simeq&GL_3(\R)/O_3(\R)Z(\R)
\ena
be the generalized Poincare upper half plane.
Given a spectral parameter $(\nu_1,\nu_2)\in\C^2$,
the  function $I_{\nu_1,\nu_2}$
on $\mk h^3$ is defined by
\bna
I_{\nu_1,\nu_2}(z)=y_1^{1+2\nu_1+\nu_2}y_2^{1+\nu_1+2\nu_2}
\ena
and the Jacquet-Whittaker function is defined by
\bna
\begin{aligned}
\mathcal W_{\nu_1,\nu_2}^{\pm}(z):&=\int_{\R^3}I_{\nu_1,\nu_2}
\left(\bma&&1\\&1\\1 \ema \bma 1&u_2&u_3\\&1&u_1\\&&1\ema z\right)e(-(u_1\pm u_2))du_1du_2du_3
\end{aligned}
\ena
where $e(x)=\exp(2\pi i x)$.

Let $\nu_3=-\nu_1-\nu_2$ and
\bea
\mu_1=\nu_1+2\nu_2,\quad \mu_2=\nu_1-\nu_2,\quad\mu_3=-2\nu_1-\nu_2.\label{Langlands-parameter}
\eea
We will simultaneously use $\bm \mu=(\mu_1,\mu_2,\mu_3)$ and $(\nu_1,\nu_2,\nu_3)$ coordinates,
\bna
\nu_1=\frac{1}{3}(\mu_1-\mu_2),\quad \nu_2=\frac{1}{3}(\mu_2-\mu_3),\quad
\nu_3=-\nu_1-\nu_2.
\ena

\subsection{Automorphic forms for $SL(3,\Z)$}
\subsubsection{Hecke-Maass cusp forms}
A Hecke-Maass cusp form for $\Gamma=SL_3(\Z)$
of type $(\frac{1}{3}+\nu_1,\frac{1}{3}+\nu_2)$
is a function $\phi:\Gamma\backslash\mk h^3\rightarrow
\C$ which
has the Fourier expansion
\bna
\phi(z)=\sum_{m_1=1}^\infty\sum_{m_2\neq 0}\frac{A_\phi(m_1,m_2)}{m_1|m_2|}
\sum_{\gamma\in U_2\backslash SL_2(\Z)}\mathcal W_{\nu_1,\nu_2}^{\mathrm{sgn}(m_2)}
\left(\bma m_1|m_2|\\&m_1\\&&1\ema\bma\gamma&\\&1\ema z\right)c_{\nu_1,\nu_2}.
\ena
Here $A_\phi(m_1,m_2)$
are eigenvalues of $\phi$ at $(m_1,m_2)$ satisfying
\bna
A_\phi(m_1,m_2)\ll m_1m_2,
\ena
$\mathcal W^{\pm}_{\nu_1,\nu_2}(z)$ is the Jacquet-Whittaker function and $c_{\nu_1,\nu_2}$
is a constant depending only on $\nu_1$ and $\nu_2$ (see formula (2.13) in \cite{Bl2013}).

Let $\bm\mu_\phi=(\mu_1,\mu_2,\mu_3)$ be the Langlands parameter of $\phi$
where $\mu_j$ are given by (\ref{Langlands-parameter}).
The $L$-function associated to $\phi$ is defined by
\bna
L(s,\phi):=\sum_{m\geq 1}\frac{A_\phi(1,m)}{m^s}
\ena
for $\re(s)>2$.
It has analytic continuation for $s\in\C$ and satisfies the functional equation
\bna
\Lambda(s,\phi)=\prod_{j=1}^3\Gamma_\R(s+\mu_j)L(s,\phi)
=\Lambda(1-s,\phi^\vee).
\ena
Here $\Gamma_\R(s)=\pi^{-\frac{s}{2}}\Gamma(s/2)$
and
$\phi^\vee$ is the dual Hecke-Maass cusp form of $\phi$ with
\bna
A_{\phi^\vee}(m_1,m_2)=A_{\phi}(m_2,m_1),\quad
\bm \mu_{\phi^\vee}=(-\mu_1,-\mu_2,-\mu_3).
\ena

\subsubsection{The minimal Eisenstein series}\label{subsubsec-Eisenstein-series-associated-to-minimal}
Let $P_{1,1,1}$ be the standard minimal parabolic subgroup of $GL_3$ and $U_3$
be the unipotent radical of $P_{1,1,1}$.
Given a spectral parameter $(\nu_1,\nu_2)\in\C^2$, let $\bm \mu=(\mu_1,\mu_2,\mu_3)$
be the Langlands parameter given by (\ref{Langlands-parameter}).
The minimal Eisenstein series
\bna
E_{\nu_1,\nu_2}^{\mathrm{min}}(z):=\sum_{\gamma\in U_3(\Z)\backslash \Gamma}I_{\nu_1,\nu_2}(\gamma z)
\ena
is defined for $\re(\nu_1)$ and $\re(\nu_2)$ sufficient large and has meromorphic continuation
to all $(\nu_1,\nu_2)\in\C^2$.
The Hecke eigenvalues  $A^{\min}_{\nu_1,\nu_2}(m,n)$
of  $E^{\mathrm{min}}_{\nu_1,\nu_2}(z)$ at $(m,n)$
are defined by
\bna
A^{\min}_{\nu_1,\nu_2}(1,n)=\sum_{d_1d_2d_3=n}d_1^{-\mu_1}d_2^{-\mu_2}d_3^{-\mu_3}
\ena
and by Hecke relations
\bna
A^{\min}_{\nu_1,\nu_2}(m,1)&=&\overline{A^{\min}_{\nu_1,\nu_2}(1,m)},\nonumber\\
A^{\min}_{\nu_1,\nu_2}(m_1,m_2)&=&\sum_{d\mid (m_1,m_2)}\mu(d)A^{\min}_{\nu_1,\nu_2}\left(\frac{m_1}{d},1\right)
A^{\min}_{\nu_1,\nu_2}\left(1,\frac{m_2}{d}\right).
\ena
The $L$-function associated to $E^{\min}_{\nu_1,\nu_2}(z)$ is
\bna
L(s,E^{\min}_{\nu_1,\nu_2}):=\sum_{m\geq 1}\frac{A^{\min}_{\nu_1,\nu_2}(1,m)}{m^s}
=\zeta(s+\mu_1)\zeta(s+\mu_2)\zeta(s+\mu_3)
\ena
where $\mu_i$ are given by (\ref{Langlands-parameter}).

\subsubsection{The Maximal Eisenstein series}
\label{subsub-Eisenstein-Series-associated-to-Maximal}
Let $g:SL_2(\Z)\backslash \mk h^2\rightarrow\C$ be a Hecke-Maass cusp form
with the spectral parameter $it_g\in i\R$ and
 Hecke eigenvalues  $\lambda_g(m)$. We assume that $g$ is normalized by
 $\|g\|=1$.
Let
$$P_{2,1}=\left[\begin{matrix} *&*&*\\ *&*&*\\&&*\end{matrix}\right]$$
be the standard maximal parabolic subgroup of $GL_3$.
For $u\in\C$, the maximal Eisenstein series
\bna
E^{\max}_{u,g}(z):=\sum_{\gamma\in P_{2,1}(\Z)\backslash\Gamma}
\det(\gamma z)^{\frac{1}{2}+u} g(\mk m_{P_{2,1}}(\gamma z))
\ena
 is defined for $\re(u)$ sufficient large.
Here $\mk m_{P_{2,1}}$
is the restriction to the upper left corner,
\bna
\mk m_{P_{2,1}}:\mk h^3\rightarrow \mk h^2,\quad
\bma y_1y_2&y_1x_2&x_3\\&y_1&x_1\\&&1\ema\mapsto\bma y_2&x_2\\&1\ema.
\ena

The Hecke eigenvalue $A_{u,g}^{\max}(m,n)$ of
 $E^{\max}_{u,g}$ at $(m,n)$ is defined by
\bea
A_{u,g}^{\max}(1,n)=\sum_{d_1d_2=|n|}\lambda_g(d)d_1^{-u}d_2^{2u}\label{Fourier-coef-maximal-Eisenstein-1}
\eea
and by the Hecke relations
\bea
A^{\max}_{u,g}(m,1)&=&\overline{A^{\max}_{u,g}(1,m)},\nonumber\\
A^{\max}_{u,g}(m_1,m_2)&=&\sum_{d\mid (m_1,m_2)}\mu(d)
A_{u,g}^{\max}\left(\frac{m_1}{d},1\right)
A_{u,g}^{\max}\left(1,\frac{m_2}{d}\right)\label{Fourier-coef-maximal-Eisenstein-2}.
\eea
The $L$-function associated to $E^{\max}_{u,g}(z)$ is
\bna
L(s,E^{\max}_{u,g})=\sum_{m\geq 1}\frac{A_{u,g}^{\max}(1,m)}{m^s}
=\zeta(s-2u)L(s+u,g)
\ena
and the complete $L$-function is
\bna
\Lambda(s,E^{\max}_{u,g})=\prod_{i=1}^3\Gamma_\R(s+\mu_i')L(s,E^{\max}_{u,g})
=\Lambda(1-s,E^{\max}_{-u,g}),
\ena
where
\bea
\mu_1'=u+it_g,\quad\mu_2'=u-it_g,\quad\mu_3'=-2u.
\label{Langlands-parameter-Maximal-Eisenstein}
\eea
\subsection{The Kuznetsov trace formula}\label{subsec-kuznetsov}
We recall the Kuznetsov trace formula in the version of \cite{Bu2014}.
Let $d\bm\mu=d\mu_1d\mu_2$ be the standard measure
on the Lie algebra
\bna
\Lambda_\infty:=\{\bm\mu=(\mu_1,\mu_2,\mu_3)\in\C^3,\quad \mu_1+\mu_2+\mu_3=0\}.
\ena
We set $d_{\mathrm{spec}}(\bm \mu)=\mathrm{spec}(\bm\mu)d\bm\mu$ with
\bna
\mathrm{spec}(\bm \mu):=\prod_{j=1}^3\left(3\nu_j\tan\left(\frac{3\pi}{2}\nu_j \right)\right).
\ena

\subsubsection{Normalized factors}
The normalized factors are defined as follows.
\bit
\item For $\phi$ a Hecke-Maass cusp form with $\bm \mu_\phi=(\mu_1,\mu_2,\mu_3)$,
we denote by
\bea
\mathcal N_\phi:=\|\phi\|^2\prod_{j=1}^3
\cos\left(\frac{3}{2}\pi \nu_{j}\right).\label{normalized-factor-cusp}
\eea
Note that for $\bm\mu_\phi=(\mu_1,\mu_2,\mu_3)$ with $\mu_i\asymp T$, one has
\bna
\mathcal N_\phi\asymp \res_{s=1}L(s,\phi\otimes\phi^\vee)\ll T^\epsilon.
\ena

\item
 For $E^{\mathrm{min}}_{\nu_1,\nu_2}(z)$ the minimal Eisenstein series
 with the Langlands parameter $\bm\mu(E^{\min}_{\nu_1,\nu_2})=(\mu_1,\mu_2,\mu_3)$,
the normalized factor is defined by
\bna
\mathcal N^{\min}_{\nu_1,\nu_2}:=\frac{1}{16}\prod_{j=1}^3|\zeta(1+3\nu_{j})|^2.
\ena

\item
For $E^{\max}_{u,g}(z)$ the maximal Eisenstein series, we define
\bna
\mathcal N^{\max}_{u,g}:=8L(1,\mathrm{Ad}^2g)|L(1+3u,g)|^2.
\ena
\eit

\subsubsection{Kloosteman Sums}
Two type of Kloosterman sums are defined as follows.
Assume $D_1\mid D_2$, we have the incomplete Kloosterman sum
\bna
\tilde S(n_1,n_2,m_1,D_1,D_2):=
\sum_{C_1(\bmod D_1),C_2(\bmod D_2)\atop{(C_1,D_1)=(C_2,D_2/D_1)=1}}
e\left(n_2\frac{\overline C_1 C_2}{D_1}+m_1\frac{\overline C_2}{D_2/D_1}+n_1\frac{C_1}{D_1}\right).
\ena
The complete Kloosterman sum is defined by
\bna
&&S(n_1,n_2,m_1,m_2,D_1,D_2)\\
&&:=\sum_{B_1,C_1\bmod D_1\atop{B_2,C_2\bmod D_2\atop{D_1C_2+B_1B_2+D_2C_1=0\bmod D_1D_2\atop{(B_j,C_j,D_j)=1}}}}
e\left(\frac{n_1B_1+m_1(Y_1D_2-Z_1B_2)}{D_1}+\frac{m_2B_2+n_2(Y_2D_1-Z_2B_1)}{D_2}\right)
\ena
where $Y_jB_j+Z_jC_j\equiv 1\bmod D_j$ for $j=1,2$.

By the standard (Weil-type) bounds we have (see formulas 3.1 and 3.2 in \cite{BB2015})
\bna
\tilde S(n_1,n_2,m_1,D_1,D_2)\ll\left((m_1,D_2/D_1)D_1^2,(n_1,n_2,D_1)D_2)\right)(D_1D_2)^\epsilon
\ena
and
\bna
 S(n_1,n_2,m_1,m_2,D_1,D_2)\ll(D_1D_2)^{1/2+\epsilon}
 \left\{(D_1,D_2)(m_1n_1,[D_1,D_2])(m_2n_2,[D_1,D_2])\right\}^{1/2}.
\ena
\subsubsection{Integral kernels }
Following Theorems 2 and 3 in \cite{Bu2014},
the integral kernels are given as follows.
For $s\in\C$ and $\bm \mu=(\mu_1,\mu_2,\mu_3)$, we let
\bna
\tilde G^{\pm}(s,\bm\mu):=\frac{\pi^{-3s}}{12288\pi^{7/2}}
\left(\prod_{j=1}^3\frac{\Gamma(\frac{s-\mu_j}{2})}{\Gamma(\frac{1-s+\mu_j}{2})}
\pm i\prod_{j=1}^3\frac{\Gamma\left(\frac{1+s-\mu_j}{2}\right)}
{\Gamma\left(\frac{2-s+\mu_j}{2}\right)}\right).
\ena
The integral kernel associated to $w_4$ is defined by
\bna
K_{w_4}(y;\bm\mu)=\int_{-i\infty}^{i\infty}|y|^{-s}\tilde G^\epsilon(s,\bm\mu)\frac{ds}{2\pi i}
\ena
for $y\in \R-\{0 \}$ with $\epsilon=\mathrm{sgn}(y)$.

For $(s_1,s_2)\in\C^2$ and $\bm \mu=(\mu_1,\mu_2,\mu_3)$, we let
\bna
G(s_1,s_2,\bm\mu):=\frac{1}{\Gamma(s_1+s_2)}\prod_{j=1}^3\Gamma(s_1-\mu_j)\Gamma(s_2+\mu_j).
\ena
We also define the following trigonometric functions
\bna
S^{++}(s_1,s_2;\bm\mu)&=&\frac{1}{24\pi^2}\prod_{j=1}^3\cos\left(\frac{3}{2}\pi\nu_j\right),\\
S^{+-}(s_1,s_2;\bm\mu)&=&-\frac{1}{32\pi^2}\frac{\cos(\frac{3}{2}\pi\nu_2)\sin(\pi(s_1-\mu_1))
\sin(\pi(s_2+\mu_2))\sin(\pi(s_2+\mu_3))}
{\sin(\frac{3}{2}\pi\nu_1)\sin(\frac{3}{2}\pi\nu_3)\sin(\pi(s_1+s_2))},\\
S^{-+}(s_1,s_2;\bm\mu)&=&
-\frac{1}{32\pi^2}\frac{\cos(\frac{3}{2}\pi\nu_1)\sin(\pi(s_1-\mu_1))\sin(\pi(s_1-\mu_2))
\sin(\pi(s_2+\mu_3))}
{\sin(\frac{3}{2}\pi\nu_2)\sin(\frac{3}{2}\pi\nu_3)\sin(\pi(s_1+s_2))},\\
S^{--}(s_1,s_2;\bm\mu)&=&\frac{1}{32\pi^2}
\frac{\cos(\frac{3}{2}\pi\nu_3)\sin(\pi(s_1-\mu_2))\sin(\pi(s_2+\mu_2))}
{\sin(\frac{3}{2}\pi\nu_2)\sin(\frac{3}{2}\pi\nu_1)}.
\ena
The integral kernel associated to the longest Weyl's element $w_l$ is defined by
\bna
K_{w_l}^{\epsilon_1,\epsilon_2}(y_1,y_2;\bm\mu)
=\int_{-i\infty}^{i\infty}\int_{-i\infty}^{i\infty}
|4\pi^2y_1|^{-s_1}|4\pi^2y_2|^{-s_2}G(s_1,s_2;\bm\mu)S^{\epsilon_1,\epsilon_2}(s_1,s_2;\bm\mu)\frac{ds_1ds_2}{(2\pi i)^2}
\ena
for $(y_1,y_2)\in (\R-\{0\})^2$ with $\epsilon_i=\mathrm{sgn}(y_i)$.

\subsubsection{The Kuznetsov's trace formula}
Let $n_1,n_2,m_1,m_2\in \mathbb N$ and let $h(\bm \mu)$ be a function
that is holomorphic on
\bna
\Lambda_{1/2+\delta}=\left\{\bm\mu=(\mu_1,\mu_2,\mu_3)\in\C^3,\quad \mu_1+\mu_2+\mu_3=0,\re(\mu_j)\leq\frac{1}{2}+\delta\right\}
\ena
for some $\delta>0$, symmetric under the Weyl group $W$, rapidly decaying as $|\im \mu_j|\rightarrow\infty$
and satisfies
\bna
h(3\nu_1+1,3\nu_2+1,3\nu_3+1)=0.
\ena
Then one has
\bna
\mathcal C+\mathcal E_{min}+\mathcal E_{max}
=\Delta+\Sigma_4+\Sigma_5+\Sigma_l,
\ena
where
\bna
\mathcal C
&=&\sum_{\phi}\frac{h(\bm{\mu}_\phi)}{\mathcal N_\phi}A_\phi(n_1,n_2)\overline{A_\phi(m_1,m_2)},\\
\mathcal E_{max}
&=&
\frac{1}{2\pi i}\sum_{g}\int_{\re(u)=0}
\frac{h(u+it_g,u-it_g,-2u)}{\mathcal N^{\max}_{u,g}}
A^{\max}_{u,g}(n_1,n_2)\overline{A^{\max}_{u,g}(m_1,m_2)}du,\\
\mathcal E_{min}&=&
\frac{1}{24 (2\pi i)^2}
\iint_{\re(\bm\mu)=0}\frac{h(\bm{\mu})}{\mathcal N^{\min}_{\nu_1,\nu_2}}
A^{\min}_{\bm \mu}(n_1,n_2)\overline{A^{\min}_{\bm\mu}(m_1,m_2)}d\bm\mu,
\ena
and
\bna
\Delta&=&\delta_{m_1,n_1}\delta_{m_2,n_2}\frac{1}{192\pi^5}
\iint_{\re(\bm\mu)=0}h(\bm\mu) d_\mathrm {spec}\bm\mu,\\
\Sigma_4&=&\sum_{\epsilon\in\{\pm 1\}}\sum_{D_2\mid D_1\atop{m_2D_1=n_1D_2^2}}
\frac{\tilde S(-\epsilon n_2,m_2,m_1,D_2,D_1)}{D_1D_2}\Phi_{w_4}
\left(\frac{\epsilon m_1m_2n_2}{D_1D_2};h\right),\\
\Sigma_5&=&\sum_{\epsilon\in\{\pm 1\}}\sum_{D_1\mid D_2\atop{m_1D_2=n_2D_1^2}}
\frac{\tilde S(-\epsilon n_1,m_1,m_2,D_1,D_2)}{D_1D_2}\Phi_{w_5}
\left(\frac{\epsilon n_1m_1m_2}{D_1D_2};h\right),\\
\Sigma_l&=&\sum_{\epsilon_1,\epsilon_2\in\{\pm 1\}}\sum_{D_1,D_2}
\frac{S(\epsilon_2 n_2,\epsilon_1n_1;m_1,m_2;D_1,D_2)}{D_1D_2}\Phi_{w_l}
\left(-\frac{\epsilon_2 m_1n_2D_2}{D_1^2},-\frac{\epsilon_1 m_2n_1D_1}{D_2^2};h\right).
\ena
Here
\bna
\Phi_{w_4}(y;h)&=&\iint_{\re(\bm\mu)=0} h(\bm\mu)K_{w_4}(y;\bm\mu)
d_\mathrm{spec}\bm\mu,\\
\Phi_{w_5}(y;h)&=&\iint_{\re(\bm\mu)=0} h(\bm\mu)K_{w_4}(-y;-\bm\mu)d_{\mathrm{spec}}\bm\mu,\\
\Phi_{w_l}(y_1,y_2;h)&=&\iint_{\re(\bm\mu)=0} h(\bm\mu)K_{w_l}^{\mathrm{sgn}(y_1),\mathrm{sgn}(y_2)}
(y_1,y_2;\bm\mu)d_{\mathrm{spec}}\bm\mu.
\ena

\subsection{The choice of the test function}\label{subsec-choice-testfunc}

By unitarity and the Jacquet-Shalika's bounds, the Langlands parameter $\bm\mu_\phi$ of a Hecke-Maass cusp form $\phi$
for $SL_3(\Z)$ is contained in
\bna
\Lambda_{1/2}':=\left\{(\mu_1,\mu_2,\mu_3)\in\C^3,\quad
\begin{aligned}
&| \re(\mu_j)|\leq \frac{1}{2},\quad  \mu_1+\mu_2+\mu_3=0,\\
&\{\mu_1,\mu_2,\mu_3\}=\{\overline\mu_1,\overline\mu_2,\overline\mu_3\}
\end{aligned}
 \right\}.
\ena
 Let $\bm\mu^0=(\mu^0_1,\mu^0_2,\mu^0_3)$ be in generic position in $\Lambda'_{1/2}$, i.e.
\bna
|\mu_{j}^0|\asymp \|\bm\mu^0\|:=T,\quad 1\leq j\leq 3.
\ena
Following \cite{BB2015}(or see \cite{HLZ2017}),
we choose a test function $h(\bm\mu)$ to localizes at a ball of radius $M=T^\theta$ with $0<\theta<1$
about $w(\bm\mu^0)$ for each $w\in W$. It is defined by
\bna
h(\bm\mu):=P(\bm\mu)^2\left(\sum_{w\in W}\psi\left(\frac{w(\bm\mu)-\bm\mu^0}{M}\right)\right)^2,
\ena
where $\psi(\bm\mu)=\exp\left(\mu_1^2+\mu_2^2+\mu_3^2\right)$ and
\bna
P(\bm\mu)=\prod_{0\leq n\leq A_0}
\prod_{j=1}^3\frac{\left(\nu_j-\frac{1}{3}(1+2n)\right)\left(\nu_j+\frac{1}{3}(1+2n)\right)}{|\nu^0_j|^2}
\ena
for some fixed large $A_0>0$. Here
\bna
W&=&\left\{I,w_2=\bma1&\\&&1\\&1\ema, w_3=\bma &1\\1\\&&1\ema,w_4=\bma&1\\&&1\\1\ema,\right.\\
&&\quad\left.
w_5=\bma &&1\\1&\\&1\ema, w_l=\bma&&1\\&1\\1\ema\right\}
\ena
is the Weyl group for $GL_3(\R)$.

\medskip

We need the following two lemmas in \cite{BB2015}, which are used in truncating
summations in geometric terms after the application of  the Kuznetsov's trace formula.
\begin{lemma}\label{lemma-truncation-1}
Let $0<|y|\leq T^{3-\epsilon}$. Then for any constant $A\geq 0$ one has
\bna
\Phi_{w_4}(y;h)\ll_{\epsilon,B}T^{-A}.
\ena
If $|y|>T^{3-\epsilon}$ then
\bna
|y|^j\Phi_{w_4}^{(j)}(y;h)\ll_{j,\epsilon}T^{3}M^2(T+|y|^{1/3})^j
\ena
for any $j\in\mathbb N_0$.
\end{lemma}

\begin{lemma}\label{lemma-truncation-2}
Let $\mathcal Y:=\min\{|y_1|^{1/3}|y_2|^{1/6},|y_1|^{1/6}|y_2|^{1/3}\}$. If $\mathcal Y\leq T^{1-\epsilon}$,
then
\bna
\Phi_{w_l}(y_1,y_2;h)\ll_{B,\epsilon} T^{-A}
\ena
for any fixed constant $A\geq 0$. If $\mathcal Y\gg T^{1-\epsilon}$, then
\bna
&&|y_1|^{j_1}|y_2|^{j_2}\frac{\partial^{j_1}}{\partial y_1^{j_1}}\frac{\partial^{j_2}}{\partial y_1^{j_2}}
\Phi_{w_l}(y_1,y_2)\\
&\ll_{j_1,j_2,\epsilon}&T^{3}M^2(T+|y_1|^{1/2}+|y_1|^{1/3}|y_2|^{1/6})^{j_1}
(T+|y_2|^{1/2}+|y_2|^{1/3}|y_1|^{1/6})^{j_2}
\ena
for all $j_1,j_2\in\mathbb N_0$.
\end{lemma}

\subsection{Rankin-Selberg $L$-functions}
We recall holomorphic Hecke cusp forms in \cite{Iw1997}.
Let $f$ be a normalized holomorphic Hecke cusp form of weight $k$ for $SL_2(\Z)$
such that $f$ has the Fourier expansion
\bna
f(z)=\sum_{m\geq 1}\lambda_f(m)m^{\frac{k-1}{2}}e(mz),
\ena
where  $\lambda_f(m)$ are
Hecke eigenvalues of the Hecke operators $T(m)$.
The $L$-function associated to $f$
is
\bna
L(s,f)=\sum_{m\geq 1}\frac{\lambda_f(m)}{m^s}
\ena
which is absolutely convergent for $\re(s)>1$ by the
Ramanujan-Deligne's bound $\lambda_f(m)\ll m^\epsilon$.
It has analytic continuation for all $s\in\C$ and satisfies the functional equation
\bna
\Lambda(s,f):=\Gamma_\R\left(s+\frac{k-1}{2}\right)\Gamma_\R\left(s+\frac{k+1}{2}\right)L(s,f)
=i^k\Lambda(1-s,f).
\ena

Let $f$ be as above and $\phi$ be a Hecke-Maass cusp form for $SL_3(\Z)$
with Langlands parameter $\bm\mu_\phi=(\mu_1,\mu_2,\mu_3)$.
The Rankin-Selberg $L$-function $L(s,f\otimes \phi)$ is defined by (see Section 12.2 in \cite{Go2006})
\bna
L(s,f\otimes \phi):=\sum_{m_1\geq 1}\sum_{m_2\geq 1}\frac{\lambda_f(m_2)\overline{A_\phi(m_1,m_2)}}{(m_1^2m_2)^s}
\ena
for $\re(s)$ sufficient large. It has analytic continuation for
all $s\in\C$ and satisfies the functional equation
\bna
\begin{aligned}
\Lambda(s,f\otimes\phi)&=\prod_{i=1}^3\Gamma_\R\left(s+\frac{k-1}{2}-\mu_i\right)
\Gamma_\R\left(s+\frac{k+1}{2}-\mu_i\right)
L(s,f\otimes\phi)\\
&=(i^{k})^3\Lambda(1-s,f\otimes\phi^\vee),
\end{aligned}
\ena
where $\phi^\vee$ is the dual Maass cusp form associated to $\phi$.

Let $E^{\min}_{\nu_1,\nu_2}(z)$ be the minimal Eisenstein series with
the Langlands parameter $\bm\mu(E^{\min}_{\nu_1,\nu_2})$.
By Euler products of $L(s,f)$ and $L(s,E^{\min}_{\nu_1,\nu_2})$,
we have
\bna
\begin{aligned}
L(s,f\otimes E^{\min}_{\nu_1,\nu_2}):&
=\sum_{m_1\geq 1}\sum_{m_2\geq 1}\frac{\lambda_f(m_2)\overline{A^{\min}_{\nu_1,\nu_2}(m_1,m_2)}}{(m_1^2m_2)^s}\\
&=L(s-\mu_1,f)L(s-\mu_2,f)L(s-\mu_3,f).
\end{aligned}
\ena
It satisfies the functional equation
\bna
\begin{aligned}
\Lambda(s,f\otimes E^{\min}_{\nu_1,\nu_2}):&=\prod_{j=1}^3\Gamma_\R\left(s+\frac{k-1}{2}-\mu_j\right)
\Gamma_\R\left(s+\frac{k+1}{2}-\mu_j\right)L(s,f\times E^{\min}_{\nu_1,\nu_2})\\
&=\Lambda(1-s,f\otimes E^{\min}_{-\nu_1,-\nu_2}).
\end{aligned}
\ena

For $E^{\max}_{u,g}(z)$ the maximal Eisenstein series with
 $\bm\mu(E^{\max}_{u,g})=(\mu_1',\mu_2',\mu_3')$
where $\mu_j'$ are given by (\ref{Langlands-parameter-Maximal-Eisenstein}),
we have
\bna
\begin{aligned}
L(s,f\otimes E^{\max}_{u,g}):&=
\sum_{m_1\geq 1}\sum_{m_2\geq 1}\frac{\lambda_f(m_2)\overline{A^{\max}_{\nu,u}(m_1,m_2)}}{(m_1^2m_2)^s}\\
&=L(s+2u,f)L(s-u,f\otimes g),
\end{aligned}
\ena
where
$L(s,f\otimes g)$
is the Rankin-Selberg function associated to $f$ and $g$.
The complete $L$-function is
\bna
\begin{aligned}
\Lambda(s,f\otimes E^{\max}_{u,g})&=\prod_{j=1}^3\Gamma_\R\left(s+\frac{k-1}{2}-\mu_j'\right)
\left(s+\frac{k+1}{2}-\mu_j'\right)
L(s,f\otimes E^{\max}_{u,g})\\
&=i^k\Lambda(1-s,f\otimes E^{\max}_{-u,g}).
\end{aligned}
\ena

\subsection{The approximate functional equation}
For the Rankin-Selberg $L$-function defined in the previous section,
we have the following
approximate functional equation (see Theorem 5.3 in \cite{IK2004}).
\begin{lemma}\label{lemma-approximate-func}Let $G(s)=e^{s^2}$. We have
\bna
L\left(\frac{1}{2},f\otimes\phi\right)&=&\sum_{m_1\geq 1}\sum_{m_2\geq 1}
\frac{\lambda_f(m_2)A_\phi(m_2,m_1)}{(m_1^2m_2)^{1/2}}
V_k(m_1^2m_2,\bm{\mu}_\phi)\\
&&+i^{k}\sum_{m_1\geq 1}\sum_{m_2\geq 1}
\frac{\lambda_f(m_1)A_\phi(m_1,m_2)}{(m_1^2m_2)^{1/2}}
\tilde V(m_1^2m_2;k,\bm{\mu}_\phi),
\ena
where
\bea
V_k(y,\bm{\mu})=\frac{1}{2\pi i}\int_{(3)}y^{-s}
\prod_{i=1}^3\frac{\Gamma_\R\left(s+\frac{1}{2}+\frac{k-1}{2}-\mu_i\right)
\Gamma_\R\left(s+\frac{1}{2}+\frac{k+1}{2}-\mu_i\right)}
{\Gamma_\R\left(\frac{1}{2}+\frac{k-1}{2}-\mu_i\right)
\Gamma_\R\left(\frac{1}{2}+\frac{k+1}{2}-\mu_i\right)}G(s)\frac{ds}{s}
\label{V-y-k-mu}
\eea
and
\bna
\tilde V_k(y,\bm{\mu})=\frac{1}{2\pi i}\int_{(3)}y^{-s}
\prod_{i=1}^3\frac{\Gamma_\R\left(s+\frac{1}{2}+\frac{k-1}{2}+\mu_i\right)
\Gamma_\R\left(s+\frac{1}{2}+\frac{k+1}{2}+\mu_i\right)}
{\Gamma_\R\left(\frac{1}{2}+\frac{k-1}{2}-\mu_i\right)
\Gamma_\R\left(\frac{1}{2}+\frac{k+1}{2}-\mu_i\right)}G(s)\frac{ds}{s}.
\ena
\end{lemma}
The functions $V_k(y,\bm\mu)$ and $\tilde V_k(y,\bm\mu)$ have the following properties,
which can be proved by the method in Proposition 5.4 in \cite{IK2004}.
\begin{lemma}\label{lemma-proper-V-k-mu}
Assume that $\bm\mu=(\mu_1,\mu_2,\mu_3)$ with $\mu_i\asymp T$.
One has
\bea
\nonumber y^a\frac{\partial^{a}}{\partial y^a}V_k(y,\bm\mu)\ll_k \left(\frac{y}{T^3}\right)^{-A},\quad
y^a\frac{\partial^{a}}{\partial y^a}\tilde V_k(y,\bm\mu)\ll_k \left(\frac{y}{T^3}\right)^{-A}
\eea
for any large number $A>0$ and any $a\in\mathbb N_0$.
Moreover, for $y\gg T^{3}$,
\bna
V_k(y,\bm\mu)
&=&1+O_{B,k}\left(\frac{T^3}{y}\right)^{-B}\\
\tilde V_k(y,\bm\mu)
&=&\prod_{i=1}^3\frac{\Gamma(\frac{k}{2}+\mu_i)}
{\Gamma(\frac{k}{2}-\mu_i)}+O_{B,k}\left(\frac{T^3}{y}\right)^{-B}
\ena
for any $0<B<\frac{k-1}{2}$.
\end{lemma}

\section{Proof of Theorem \ref{thm-main}}
Let $k\equiv 0\bmod 4$. For $h(\bm\mu)$ defined in section \ref{subsec-choice-testfunc},
we consider
\bna
\mathcal A=\sum_{\phi}\frac{h(\bm \mu_\phi)}{\mathcal N_\phi}
A_\phi(p,p)L(1/2,f\otimes\phi),
\ena
where  $\phi$ runs over a  Hecke-Maass basis of the space of
Maass cusp forms for $SL_3(\Z)$.
By the approximate functional equation in Lemma \ref{lemma-approximate-func}, one has
\bna
\mathcal A=\mathcal A_1+\mathcal A_2
\ena
where
\bna
\mathcal A_1&=&
\sum_{m_1\geq 1}\sum_{m_2\geq 1}\frac{\lambda_f(m_2)}{(m_1^2m_2)^{1/2}}
\sum_\phi\frac{h(\bm\mu_\phi)V_k(m_1^2m_2,\bm{\mu_\phi})}
{\mathcal N_j}A_{\phi}(m_2,m_1){A_\phi(p,p)},\\
\mathcal A_2&=&\sum_{m_1\geq 1}\sum_{m_2\geq 1}\frac{\lambda_f(m_2)}{(m_1^2m_2)^{1/2}}
\sum_\phi\frac{h(\bm\mu_\phi)\tilde V_k(m_1^2m_2,\bm{\mu_\phi})}
{\mathcal N_j}A_{\phi}(m_1,m_2){A_\phi(p,p)}.
\ena
Thus Theorem 1.1 follows from
\bea
\mathcal A_1&=&
\frac{\lambda_f(p)}{p^{3/2}}\frac{1}{192\pi^5}
\iint_{\re(\bm\mu)=0}h(\bm \mu)d_{\mathrm{spec}}(\bm\mu)
+
O_{k,\epsilon}(p^{\frac{7}{32}+\epsilon}T^{\frac{5}{2}+\epsilon}M^2)
\label{Estimat-A1},\\
\mathcal A_2&=&
\frac{\lambda_f(p)}{p^{3/2}}\frac{1}{192\pi^5}
\iint_{\re(\bm\mu)=0}h(\bm \mu)\prod_{j=1}^3\frac{\Gamma(\frac{k}{2}+\mu_j)}
{\Gamma(\frac{k}{2}-\mu_j)}d_{\mathrm{spec}}(\bm\mu)
+
O_{k,\epsilon}(p^{\frac{7}{32}+\epsilon}T^{\frac{5}{2}+\epsilon}M^2).
\label{Estimat-A2}
\eea
Since the proof of (\ref{Estimat-A2}) is the same as that of (\ref{Estimat-A1}).
We only prove (\ref{Estimat-A1}).
\medskip

For $\mathcal A_1$,
by letting
\bna
 H_{y}(\bm \mu):=h(\bm \mu)V_k(y;\bm\mu)
\ena
 and applying the Kunzetsov's trace formula in section \ref{subsec-kuznetsov}, one has
\bna
\mathcal A_1=\mathcal D_1+\mathcal R_{1,w_4}+\mathcal R_{1,w_l}-\mathcal E_{1,\max}-\mathcal E_{1,\min},
\ena
where
\bna
\mathcal D_1
&=&\frac{\lambda_f(p)}{p^{3/2}}\frac{1}{192\pi^5}
\iint_{\re(\bm\mu)=0}h(\bm\mu)V_k(p^3,\bm{\mu})d_{\mathrm{spec}}\bm\mu,\\
\mathcal R_{1,w_4}
&=&\sum_{m_1\geq 1}\sum_{m_2\geq 1}\frac{\lambda_f(m_2)}{(m_1^2m_2)^{1/2}}
\sum_{\epsilon\in\{\pm 1\}}\sum_{D_2\mid D_1\atop{pD_1=m_2D_2^2}}
\frac{\tilde S(-\epsilon m_1,p,p;D_2,D_1)}{D_1D_2}
\Phi_{w_4}\left(\frac{\epsilon m_1p^2}{D_1D_2};H_{m_1^2m_2}\right),\\
\mathcal R_{1,w_5}
&=&\sum_{m_1\geq 1}\sum_{m_2\geq 1}\frac{\lambda_f(m_2)}{(m_1^2m_2)^{1/2}}
\sum_{\epsilon\in\{\pm 1\}}\sum_{D_1\mid D_2\atop{pD_2=m_1D_1^2}}
\frac{\tilde S(\epsilon m_2,p,p;D_1,D_2)}{D_1D_2}\Phi_{w_5}
\left(\frac{\epsilon m_2p^2}{D_1D_2};H_{m_1^2m_2}\right),\\
\mathcal R_{1,w_l}
&=&\sum_{m_1\geq 1}\sum_{m_2\geq 1}\frac{\lambda_f(m_2)}{(m_1^2m_2)^{1/2}}
\sum_{\epsilon_1,\epsilon_2\in\{\pm 1\}}\sum_{D_1,D_2}
\frac{S(\epsilon_2 m_1,\epsilon_1 m_2,p,p;D_1,D_2)}{D_1D_2}\\
&&\qquad\times\Phi_{w_l}
\left(-\frac{\epsilon_2 pm_1D_2}{D_1^2},-\frac{\epsilon_1 pm_2D_1}{D_2^2};H_{m_1^2m_2}\right),
\ena
and
\bna
\mathcal E_{1,\max}&=&
\sum_{m_1\geq 1}\sum_{m_2\geq 1}\frac{\lambda_f(m_2)}{(m_1^2m_2)^{1/2}}
\sum_{g}\frac{1}{2\pi i}\nonumber\\
&&\quad\int_{\re(u)=0}
\frac{H_{m_1^2m_2}(u+it_g,u-it_g,-2u)}
{\mathcal N^{\max}_{u,g}}
A^{\max}_{u,g}(m_2,m_1){A^{\max}_{u,g}(p,p)}du,\\
\mathcal E_{1,\min}&=&
\sum_{m_1\geq 1}\sum_{m_2\geq 1}\frac{\lambda_f(m_2)}{(m_1^2m_2)^{1/2}}
\frac{1}{24(2\pi i)^2}\nonumber\\
&&\quad\iint_{\re(\bm\mu)=0}
\frac{H_{m_1^2m_2}(\bm\mu)}{\mathcal N^{\min}_{\nu_1,\nu_2}}A_{\nu_1,\nu_2}^{\min}(m_2,m_1)
A_{\nu_1,\nu_2}^{\min}(p,p)d\bm \mu
\ena
The main term in (\ref{Estimat-A1})
comes from the estimation on  $\mathcal D_1$ in (\ref{D-1}),
and the error term  comes from the contribution of
$\mathcal E_{1,\max}$ in (\ref{Estimat-maximal-Eisenstein}).
For $\mathcal E_{1,\min}$ and $\mathcal R_{1,w_4},\mathcal R_{1,w_5}$, $\mathcal R_{1,w_6}$,
we will show that their contribution is negligible under
the condition in Theorem \ref{thm-main}.

\subsection{Estimation on the continuous spectrum}
We consider $\mathcal E_{1,\min}$ firstly.
Note that $H_y(\bm\mu)=h(\bm \mu)V_k(y,\bm\mu)$. By the integral expression of $V_k(y,\bm\mu)$
in (\ref{V-y-k-mu})
and the fact that
\bna
\sum_{m_1,m_2\geq 1}\frac{\lambda_f(m_2)A_{\nu_1,\nu_2}^{\min}(m_2,m_1)}{(m_1^{2}m_2)^{s+\frac{1}{2}}}
=L\left(\frac{1}{2}+s-\mu_1,f\right)L\left(\frac{1}{2}+s-\mu_2,f\right)
L\left(\frac{1}{2}+s-\mu_3,f\right)
\ena
for $\re(s)=3$,
one has
\bna
\mathcal E_{1,\min}=\frac{1}{24(2\pi i)^2}\iint_{\re(\bm\mu)=0}
A^{\min}_{\nu_1,\nu_2}(p,p)\frac{h(\bm{\mu})}{\mathcal N^{\min}_{\nu_1,\nu_2}}
\mathcal I^{\min}_k(\bm\mu)d\bm\mu,
\ena
where
\bna
\mathcal I^{\min}_k(\bm\mu)
=\frac{1}{2\pi i}
\int_{(3)}G(s)\prod_{i=1}^3\frac{\Gamma_\R\left(s+\frac{1}{2}+\frac{k-1}{2}-\mu_i\right)
\Gamma_\R\left(s+\frac{1}{2}+\frac{k+1}{2}-\mu_i\right)}
{\Gamma_\R\left(\frac{1}{2}+\frac{k-1}{2}-\mu_i\right)
\Gamma_\R\left(\frac{1}{2}+\frac{k+1}{2}-\mu_i\right)}
L\left(\frac{1}{2}+s-\mu_i,f\right)
\frac{ds}{s}.
\ena
For $\mathcal I^{\min}_k(\bm\mu)$,
moving the line of integration to $\re(s)=\epsilon$ and applying the
subconvexity bound (see \cite{Go1982})
\bna
L(1/2+it,f)\ll_k (1+|t|)^{1/3+\epsilon},
\ena
one has
 \bna
\mathcal I_k^{\min}(\bm \mu)\ll_{\epsilon,k} \prod_{j=1}^3 (1+|\im(\mu_j)|)^{\frac{1}{3}+\epsilon}.
\ena
It gives that
\bna
\mathcal E_{1,\min}
\ll_{k,\epsilon}\iint_{\re(\bm\mu)=0}{A_{\nu_1,\nu_2}^{\min}(p,p)}
\frac{h(\bm{\mu})}{\mathcal N^{\min}_{\nu_1,\nu_2}}
\prod_j (1+|\im(\mu_j)|)^{\frac{1}{3}+\epsilon}
d\bm\mu.
\ena
Note that $A^{\min}_{ \nu_1,\nu_2}(p,p)=O(1)$
and
\bna
\mathcal N^{\min}_{\nu_1,\nu_2}=\frac{1}{16}\prod_{j=1}^3|\zeta(1+3\nu_{\pi,j})|^2
\gg\prod_{j=1}^3\left(\frac{1}{\log(1+3\im \nu_{\pi,j})}\right)^2.
\ena
One has
\bea
\mathcal E_{1,\min}
\ll_{k,\epsilon} T^{1+\epsilon}M^2.\label{Estimat-mini-Eisenstein}
\eea

Next we consider $\mathcal E_{1,\max}$.
By similar argument as above one has
\bna
\mathcal E_{1,\max}
=\sum_{g}\frac{1}{2\pi i}\int_{\re(u)=0}A^{\max}_{u,g}(p,p)
\frac{h(u+it_g,u-it_g,-2u)}{\mathcal N^{\max}_{u,g}}
\mathcal I^{\max}_{k}(u+it_g,u-it_g,-2u)du,
\ena
where
\bna
\mathcal I^{\max}_{k}(\bm\mu)
&=&
\frac{1}{2\pi i}\int_{(3)}
\prod_{i=1}^3\frac{\Gamma_\R\left(s+\frac{1}{2}+\frac{k-1}{2}-\mu_i\right)
\Gamma_\R\left(s+\frac{1}{2}+\frac{k+1}{2}-\mu_i\right)}
{\Gamma_\R\left(\frac{1}{2}+\frac{k-1}{2}-\mu_i\right)
\Gamma_\R\left(\frac{1}{2}+\frac{k+1}{2}-\mu_i\right)}\\
&&\qquad \quad L\left(\frac{1}{2}+s+2u,f\right)L\left(\frac{1}{2}+s-u,f\otimes g\right)
G(s)\frac{ds}{s}.
\ena
For $\mathcal I^{\max}_{k}(\bm\mu)$,
by moving the line of integration to $\re(s)=\frac{1}{2}+\epsilon$
and applying the fact that
\bna
L(1+\epsilon+2u,f)\ll 1,\quad L\left(1+\epsilon-u,f\otimes g\right)\ll1,
\ena
which follow from the Ramanujar-Deligue's bound and the property
of Rankin-Selberg $L$-functions
(see \cite{RS1996}),
one has
\bna
\mathcal I^{\max}_{k}(\bm\mu)
&\ll&
\int_{(\frac{1}{2}+\epsilon)}
\prod_{i=1}^3\frac{\Gamma_\R\left(s+\frac{1}{2}+\frac{k-1}{2}-\mu_i\right)
\Gamma_\R\left(s+\frac{1}{2}+\frac{k+1}{2}-\mu_i\right)}
{\Gamma_\R\left(\frac{1}{2}+\frac{k-1}{2}-\mu_i\right)
\Gamma_\R\left(\frac{1}{2}+\frac{k+1}{2}-\mu_i\right)}G(s)\frac{ds}{s}\\
&\ll_{k,\epsilon}&\prod_{j=1}^3\left(1+|\im \mu_j|\right)^{\frac{1}{2}+\epsilon}.
\ena
Moreover, by the definition of $A^{\max}_{u,g}(m,n)$
in (\ref{Fourier-coef-maximal-Eisenstein-1}) and (\ref{Fourier-coef-maximal-Eisenstein-2}), and the bound $\lambda_g(p)\ll p^{\frac{7}{64}+\epsilon}$
in \cite{KS2003}, one has
$A_{u,g}^{\max}(p,p)\ll p^{\frac{7}{32}+\epsilon}$.
These together with
\bna
\mathcal N_{u,g}^{\max}=
8L(1,\mathrm{Ad}^2g)|L(1+3u,g)|^2\gg \left(\frac{1}{1+\log |u|}\right)
\ena
give that
\bea
\mathcal E_{1,\max}
&\ll_{k,\epsilon}&p^{\frac{7}{32}+\epsilon}
\sum_{g}\int_{\re(u)=0}
\frac{h(u+it_g,u-it_g,-2u)}{\mathcal N_{u,g}}
\left(1+|\im u+t_g|\right)^{\frac{1}{2}+\epsilon}\nonumber\\
&&\qquad\qquad\quad
\left(1+|\im u-t_g|\right)^{\frac{1}{2}+\epsilon}
\left(1+|2\im u|\right)^{\frac{1}{2}+\epsilon}
d\mu\nonumber\\
&\ll_{k,\epsilon}&p^{\frac{7}{32}+\epsilon}T^{\frac{3}{2}+\epsilon}M
\sum_{g\atop{T-M\leq it_g\leq T+M}}1\nonumber\\
&\ll_{k,\epsilon}&p^{\frac{7}{32}+\epsilon}T^{\frac{5}{2}+\epsilon }M^2,\label{Estimat-maximal-Eisenstein}
\eea
where we have used the Weyl's law for Hecke-Mass cusp forms for $SL_2(\Z)$ (see \cite{Iw2002}).

\subsection{Estimation on the diagonal term $\mathcal D_1$}
For the diagonal term $\mathcal D_1$, by Lemma \ref{lemma-proper-V-k-mu}, we have
\bea
\mathcal D_1=\frac{\lambda_f(p)}{p^{3/2}}\frac{1}{192\pi^5}\left( 1+O_B\left(\frac{p}{T}\right)^{3B}\right)
\iint_{\re(\bm\mu)=0}h(\bm \mu)d_{\mathrm{spec}}(\bm\mu)\label{D-1}
\eea
for $0<B<\frac{k-1}{2}$.
The choice of $h(\bm\mu)$ in section \ref{subsec-choice-testfunc} gives
\bna
\iint_{\re(\bm\mu)=0}h(\bm \mu)d_{\mathrm{spec}}(\bm\mu)\asymp T^3M^2.
\ena
Recall that $k\geq 12$. By
(\ref{Estimat-mini-Eisenstein}) and
(\ref{Estimat-maximal-Eisenstein}),
$D_1$ gives the main term in (\ref{Estimat-A1})
if
\bea
T\gg_{k,\epsilon} p^{3+\frac{7}{16}+\epsilon}\label{main-condition}.
\eea

\subsection{Estimation on other geometric terms}
In this subsection, we show that the contribution from
other geometric terms are negligible.
For $\mathcal R_{1,w_4}$ and $\mathcal R_{1,w_l}$, it follows immediately from
the application of the truncation Lemmas \ref{lemma-truncation-1} and \ref{lemma-truncation-2},
respectively.
To show that $\mathcal R_{1,w_5}$ is negligible,
one needs to open the incomplete Kloosterman sum, rearrange the summation and apply
the Voronoi formula for $GL_2$.

\subsubsection{The term $\mathcal R_{1,w_4}$}
Consider $\mathcal R_{1,w_4}$ firstly.
By the property of $V_k(y;\bm\mu)$ in lemma \ref{lemma-proper-V-k-mu},
the terms in summations over $m_1$ and $m_2$ are negligible for those $m_1^2m_2> T^{3+\epsilon}$.
By Lemma \ref{lemma-truncation-1},
the contribution of terms in summations over $D_1$ and $D_2$
is negligible if
\bna
\frac{p^2m_1}{D_1D_2}=\frac{p^{3/2}m_1\sqrt{m_2}}{D_1^{3/2}}\leq T^{3-\epsilon}.
\ena
Thus one needs only to consider
\bna
\sum_{m_1,m_2\geq 1\atop{ m_1^2m_2\leq T^{3+\epsilon}}}\frac{\lambda_f(m_2)}{(m_1^2m_2)^{1/2}}
\sum_{\epsilon\in\{\pm 1\}}\sum_{D_2\mid D_1\atop{pD_1=m_2D_2^2
\atop{1\leq D_1\leq \frac{p(m_1^2m_2)^{1/3}}{T^{2-\epsilon}}}}}
\frac{\tilde S(-\epsilon m_1,p,p;D_2,D_1)}{D_1D_2}\Phi_{w_4}\left(\frac{\epsilon p^2m_1}{D_1D_2};H_{m_1^2m_2}\right).
\ena
Note that $m_1^2m_2\leq T^{3+\epsilon}$ and $1\leq D_1\leq \frac{p(m_1^2m_2)^{1/3}}{T^{2-\epsilon}}$
give $p\geq T^{1-\epsilon}$, which contradicts  with (\ref{main-condition}).
Thus these terms vanish and $\mathcal R_{1,w_4}$ is negligible.

\subsubsection{The term $\mathcal R_{1,w_l}$}
For $\mathcal R_{1,w_l}$,
by the property of $V_k(y,\bm\mu)$ in lemma \ref{lemma-proper-V-k-mu},
the terms in summations over $m_1$ and $m_2$ are
negligible for those $m_1^2m_1\leq T^{3+\epsilon}$.
Let
\bna
\mathcal Y:=p^{1/2}
\min\left\{\frac{m_1^{1/3}m_2^{1/6}}{D_1^{1/2}},
\frac{m_2^{1/3}m_1^{1/6}}{D_2^{1/2}}
\right\}.
\ena
By lemma \ref{lemma-truncation-2},
 the contribution is negligible
for those terms in summations over $D_1$ and $D_2$ satisfying $\mathcal Y\leq T ^{1-\epsilon}$.
Thus we need only to estimate
\bna
&&\sum_{m_1,m_2\geq 1\atop{m_1^2m_2\leq T^{3+\epsilon}}}
\frac{\lambda_f(m_2)}{(m_1^2m_2)^{1/2}}
\sum_{\epsilon_1,\epsilon_2\in\{\pm 1\}}
\sum_{D_1,D_2\atop{\mathcal Y> T^{1-\epsilon}}}
\frac{S(\epsilon_2 m_1,\epsilon_1 m_2,p,p;D_1,D_2)}{D_1D_2}\\
&&\qquad\qquad
\Phi_{w_l}
\left(-\frac{\epsilon_2 pm_1D_2}{D_1^2},-\frac{\epsilon_1 pm_2D_1}{D_2^2};H_{m_1^2m_2}\right).
\ena
Note that  $m_1^2m_2\leq T^{3+\epsilon}$ and $\mathcal Y>T^{1-\epsilon}$
give $p\geq T^{1-\epsilon}$,  which contradicts with  (\ref{main-condition}).
Thus these terms vanish and $\mathcal R_{1,w_l}$ is negligible.

\subsubsection{The term $\mathcal R_{w_5}$}
Consider $\mathcal R_{w_5}$.
By the similar argument in previous sections,
one needs
only to consider the contribution of
\bna
\mathcal R^*:&=&\sum_{m_1,m_2\geq 1\atop{ T^{\frac{8}{3}}\leq m_1^2m_2\leq T^{3+\epsilon}}}
\frac{\lambda_f(m_2)}{(m_1^2m_2)^{1/2}}
\sum_{\epsilon\in\{\pm 1\}}\\
&&\quad
\sum_{D_1\mid D_2\atop{pD_2=m_1D_1^2
\atop{1\leq D_2\leq \frac{p(m_2^2m_1)^{1/3}}{T^{2-\epsilon}}}}}
\frac{\tilde S(-\epsilon m_1,p,p;D_2,D_1)}{D_1D_2}\Phi_{w_4}
\left(\frac{\epsilon p^2m_1}{D_1D_2};H_{m_1^2m_2}\right),
\ena
since other terms either vanish or  are negligible.

We show that $\mathcal R^*$ is also negligible.
Recall the smooth partition of unity
\bna
1=\sum_{\alpha\geq 0}\omega\left(\frac{m_1^2m_2}{N_\alpha}\right),
\ena
where $\omega$ is a function which is smooth and compactly supported on $[\frac{1}{2},\frac{5}{2}]$
and $N_\alpha=2^\alpha$.
One has
\bna
\mathcal R^*
&\ll&\sum_{\alpha\geq 0\atop{T^{\frac{8}{3}}\ll N_\alpha\ll T^{3+\epsilon}}}
\sum_{m_1,m_2\geq 1}\omega\left(\frac{m_1^2m_2}{N_\alpha}\right)\frac{\lambda_f(m_2)}
{\left(m_1^2m_2\right)^{1/2}}
\sum_{\epsilon\in\{\pm 1\}}\\
&&\qquad\sum_{D_1\mid D_2\atop{pD_2=m_1D_1^2}}
\frac{\tilde S(\epsilon m_2,p,p;D_1,D_2)}{D_1D_2}\Phi_{w_5}
\left(\frac{\epsilon m_2p^2}{D_1D_2};H_{m_1^2m_2}\right).
\ena
Let $D_2=D_1\delta$. We open the incomplete Kloosterman sum,
 rearrange the summation and then obtain
\bea
\nonumber
\mathcal R^*
&\ll&
\sum_{T^{8/3}\ll N_\alpha\ll T^{3+\epsilon}}\sum_{m_1\geq 1}\frac{1}{m_1}
\sum_{\epsilon\in\{\pm 1\}}\sum_{\delta,D_1\geq 1\atop{p\delta=m_1D_1}}
\frac{1}{D_1^2\delta}\sum_{C_1(\bmod D_1),\,C_2(\bmod D_1\delta)\atop{(C_1,D_1)=(C_2,\delta)=1}}
e\left(\frac{p\overline C_1C_2}{D_1}+p\frac{\overline C_2}{\delta}\right)\\
&&\quad\sum_{m_2\geq 1}\omega\left(\frac{m_1^2m_2}{N_\alpha}\right)\frac{\lambda_f(m_2)}{\sqrt{m_2}}
\Phi_{w_5}\left(\frac{\epsilon p^2m_2}{D_1^2\delta},H_{m_1^2m_2}\right)
e\left(\epsilon m_2\frac{C_1}{D_1}\right).
\label{Last-term}
\eea
Thus one can apply the following $GL(2)$ Voronoi formula (see formula (4.71) in \cite{IK2004}).
\begin{lemma}\label{lemma-voronoi}
Let $c\geq 1$ and $(a,c)=1$.
Let $F$ be a smooth, compactly supported function on $\R^+$. One has
\bna
\sum_{m\geq 1}\lambda_f(m)e\left(\frac{am}{c}\right)F(m)
=\frac{1}{c}\sum_{n\geq 1}\lambda_f(n)e\left(-\frac{\overline a n}{c}\right)G(n),
\ena
where $G(y)=2\pi i^k\int_{0}^\infty F(x)J_{k-1}\left(\frac{4\pi\sqrt{xy}}{c}\right)dx$.
Here $J_{k}(y)$ is the $J$-Bessel function.
\end{lemma}
For the summation over $m_2$ in (\ref{Last-term}),
we apply the Voronoi formula in the above lemma  and obtain
\bna
&&\sum_{m_2\geq 1}\omega\left(\frac{m_1^2m_2}{N_\alpha}\right)\frac{\lambda_f(m_2)}{\sqrt{m_2}}
\Phi_{w_5}\left(\frac{\epsilon p^2m_2}{D_1^2\delta},H_{m_1^2m_2}\right)
e\left(\epsilon m_2\frac{C_1}{D_1}\right)\\
&=&\frac{1}{D_1}\sum_{m_2\geq 1}\lambda_f(m_2)
e\left(-\frac{\epsilon \overline C_1m_2}{D_1}\right)
G(m_2),
\ena
where
\bna
G(m_2)=2\pi i^k\int_0^\infty\omega\left(\frac{m_1^2x}{N_\alpha}\right)\frac{1}{x^{1/2}}
\Phi_{w_5}\left(\frac{\epsilon p^2x}{D_1^2\delta},H_{m_1^2m_2}\right)
J_{k-1}\left(\frac{4\pi\sqrt{xm_2}}{D_1}\right)dx.
\ena
\begin{lemma}\label{lemma-truncation-3}
We have
\bna
G(m_2)\ll_{j,k,\epsilon}\frac{\sqrt{N_\alpha}}{m_1}
\left(\frac{p^{1+\epsilon}}{N_\alpha^{\frac{1}{6}-\epsilon}m_2^{\frac{1}{2}}}
\right)^j
\ena
for any $j\in\mathbb{N}_0$.
\end{lemma}
\begin{proof}
For $G(m_2)$, we change the variable $t=\frac{m_1^2}{N_\alpha}x$ to obtain
\bna
G(m_2)=2\pi i^k\frac{\sqrt{N_\alpha}}{m_1}
\int_{0}^\infty\omega(t)
\Phi_{w_5}\left(\frac{\epsilon p^2N_\alpha}
{D_1^2\delta m_1^2}t,H_{m_1^2m_2}\right)
J_{k-1}\left(\frac{4\pi \sqrt{N_\alpha m_2t}}{m_1D_1}\right)\frac{dt}{\sqrt{t}}.
\ena
Let $R=\frac{4\pi\sqrt{N_\alpha m_2}}{m_1D_1}$.
By applying
the recurrence formula of the $J$-Bessel function
\bna
\frac{d}{dy}\left((R\sqrt{y})^{s+1}J_{s+1}(R\sqrt{y})\right)=\frac{R^2}{2}(R\sqrt{y})^sJ_{s}(R\sqrt{y}),
\ena
one has
\bna
G(m_2)
&=&2\pi i^k\frac{\sqrt{N_\alpha}}{m_1} \frac{1}{R^{k-1}}\frac{-2}{R^2}
\int_0^\infty
\left(t^{-\frac{k}{2}}\omega(t)\Phi_{w_5}\left(\frac{\epsilon p^2N_\alpha}{D_1^2\delta m_1^2}t,H_{m_1^2m_2}\right)
\right)'
\left(R\sqrt{t}\right)^{k}J_{k}(R\sqrt{t})
dt\\
&=&2\pi i^k\frac{\sqrt{N_\alpha}}{m_1} \frac{1}{R^{k-1}}\left(-\frac{2}{R^2}\right)^j\\
&&\quad
\int_0^\infty
\left(t^{-\frac{k}{2}}\omega(t)\Phi_{w_5}\left(\frac{\epsilon p^2N_\alpha}{D_1^2\delta m_1^2}t,H_{m_1^2m_2}\right)
\right)^{(j)}
\left(R\sqrt{t}\right)^{k+j-1}J_{k+j-1}(R\sqrt{t})
dt
\ena
for any $j\in\mathbb N_0$.
Note that $\Phi_{w_5}$ also satisfies Lemma \ref{lemma-truncation-1}
and one has
\bna
\left(\Phi_{w_5}\left(\frac{\epsilon p^2N_\alpha}{D_1^2\delta m_1^2}t,H_{m_1^2m_2}\right)
\right)^{(j)}
\ll T^3M^2
\left(\frac{p^2N_\alpha}{D_1^2\delta m_1^2}t\right)^{j\left(\frac{1}{3}+\epsilon\right)}.
\ena
It gives that
\bna
G(m_2)&\ll_{k,j,\epsilon}&
\frac{\sqrt{N_\alpha}}{m_1}
\left(\left.\left(\frac{p^2N}{D_1^2\delta m_1^2}\right)^{\left(\frac{1}{3}+\epsilon\right)}\right/R\right)^j\\
&\ll_{k,j,\epsilon}&\frac{\sqrt{N_\alpha}}{m_1}
\left(\frac{p^{\frac{2}{3}+\epsilon}}{N_\alpha^{\frac{1}{6}-\epsilon}m_2^{\frac{1}{2}}}
\left(\frac{m_1D_1}{\delta}\right)^{\frac{1}{3}}\right)^j
\ena
since $R=\frac{4\pi\sqrt{N_\alpha m_2}}{m_1D_1}$.
The lemma follows immediately from the fact that $m_1D_1=p\delta$.
\end{proof}

By lemma \ref{lemma-truncation-3}, the contribution is negligible for those terms in $\mathcal R^*$
satisfying
\bna
\frac{p^{1+\epsilon}}{N_\alpha^{\frac{1}{6}-\epsilon}m_2^{\frac{1}{2}}}\ll_{k,\epsilon} T^{-\epsilon}.
\ena
Note that $N_\alpha\gg T^{\frac{8}{3}}$.
Thus one needs only to consider terms in $\mathcal R^*$ satisfying the condition
\bna
p^{1+\epsilon}\gg_{k,\epsilon} N_\alpha^{\frac{1}{6}-\epsilon}T^{-\epsilon}\gg T^{\frac{4}{9}-\epsilon},
\ena
which contradicts with (\ref{main-condition}).
Thus the contribution of $\mathcal R^*$
is negligible.


\begin{thebibliography}{100}


\bibitem[Bl2013]{Bl2013}
\newblock V. Blomer,
\newblock {\it Applications of the Kuznetsov formula on $GL(3)$},
\newblock Inventiones Mathematicae 194(3), 673-729, 2013.

\bibitem[BB2015]{BB2015}
\newblock V. Blomer and J. Buttcane,
\newblock {\it On the subconvexity problem for L-functions on $GL(3)$},
\newblock arXiv:1504.02667, 2015.

\bibitem[Bu2014]{Bu2014}
\newblock J. Buttcane,
\newblock {\it The Spectral Kuznetsov Formula on $SL(3)$},
\newblock Transactions of the American Mathematical Society, 2014.

\bibitem[CD2005]{CD2005}
    \newblock G. Chinta and A. Diaconu,
    \newblock \emph{Determination of a $GL_3$ cuspform by twists of central $L$-values},
    \newblock Int. Math. Res. Notices 2005, 2941-2967, 2005.


\bibitem[GHS2009]{GHS2009}
    \newblock S. Ganguly, J. Hoffstein and J. Sengupta,
    \newblock \emph{Determining modular forms on $SL_2(\mathbb{Z})$ by central values of convolution
    $L$-functions},
    \newblock Math. Ann. 345, 843-857, 2009.

\bibitem[Go2006]{Go2006}
\newblock D. Goldfeld,
\newblock {\it Automorphic forms and L-functions for the group $GL(n,\R)$},
\newblock Cambridge University Press, 2006.

\bibitem[Go1982]{Go1982}
\newblock A. Good,
\newblock {\it The square mean of Dirichlet series associated with cusp forms},
\newblock Mathematika 29 (2), 278-295, 1982.



\bibitem[HLZ2017]{HLZ2017}
\newblock B. Huang, S. Liu and Z. Xu,
\newblock {\it Mollification and non-vanishing of automorphic $L$-functions on GL(3)},
\newblock arXiv:1704.00314, 2017.

\bibitem[Iw1997]{Iw1997}
    \newblock H. Iwaniec,
    \newblock \emph{Topics in Classical Automorphic Forms},
    \newblock Graduate Studies in Mathematics, volume 17,
    \newblock American Mathematical Society, Providence, RI, 1997.

\bibitem[Iw2002]{Iw2002}
    \newblock H. Iwaniec,
    \newblock \emph{Spectral methods of automorphic forms},
    \newblock Second edition. Graduate Studies in Mathematics, volume 53.
    \newblock American Mathematical Society, Providence, RI; Revista Mathemtica Iberoamericana, Madrid, 2002.


\bibitem[IK2004]{IK2004}
    \newblock H. Iwaniec and E. Kowalski,
    \newblock \emph{Analytic number theory},
    \newblock American Mathematical Society Colloquium Publications 53,
    \newblock American Mathematical Society, Providence, RI, 2004.

\bibitem[KS2003]{KS2003}
    \newblock H. Kim and P. Sarnak,
    \newblock \emph{Appendix: Refined estimates towards the Ramanujan and Selberg conjectures},
    \newblock J. Amer. Math. Soc 16.1, 175-181, 2003.

\bibitem[Li2007]{Li2007}
    \newblock J. Li,
    \newblock \emph{Determination of a $GL_2$ automorphic cuspidal representation by twists of
    critical $L$-values},
    \newblock  J. Number Theory 123 (2), 255-289, 2007.

\bibitem[Li2009]{Li2009}
\newblock X. Li,
\newblock {\it The Central value of the Rankin¨CSelberg $L$-Functions},
\newblock Geometric \& Functional Analysis 18.5, 1660-1695, 2009.

\bibitem[Liu2010]{Liu2010}
    \newblock S. C. Liu,
    \newblock \emph{Determination of $GL(3)$ cusp forms by central values of $GL(3)\times GL(2)$ $L$-functions},
    \newblock Int. Math. Res. Notices 2010 (21), 4025-4041, 2010.

\bibitem[Liu2011]{Liu2011}
    \newblock S. C. Liu,
    \newblock \emph{Determination of $GL(3)$ cusp forms by central values of $GL(3)\times GL(2)$ L-functions,
    level aspect},
    \newblock J. Number Theory 133 (8), 1397-1408, 2011.

\bibitem[Lu1999]{Lu1999}
    \newblock W. Luo,
    \newblock \emph{Special $L$-values of Rankin-Selberg convolutons},
    \newblock Math. Ann. 314 (3), 591-600, 1999.

\bibitem[LR1997]{LR1997}
    \newblock W. Luo and D. Ramakrishnan,
    \newblock \emph{Determination of modular forms by twists of critical $L$-values},
    \newblock Invent. Math. 130 (2), 371-398, 1997.


\bibitem[Ma2014]{Ma2014}
\newblock R. Matsuda,
\newblock {\it Determination of $GL(3)$ Hecke-Maass forms from twisted central values},
\newblock Mathematics 148.B11, 272¨C287, 2014.

\bibitem[Mu2010]{Mu2010}
    \newblock R. Munshi,
    \newblock \emph{On effective determination of modular forms by twists of critical L-values},
    \newblock Math. Ann. 347 (4), 963-978, 2010.

\bibitem[MS2015]{MS2015}
    \newblock R. Munshi and J. Sengupta,
    \newblock \emph{On effective determination of Maass forms from central values of Rankin-Selberg L-function},
    \newblock Forum Mathematicum 27(1), 467-484, 2015.


\bibitem[Pi2010]{Pi2010}
    \newblock Q. Pi,
    \newblock \emph{Determining cusp forms by central values of Rankin-Selberg L-functions},
    \newblock J. Number Theory 130(10), 2283-2292, 2010.

\bibitem[Pi2014]{Pi2014}
\newblock Q. Pi
\newblock {\it On effective determination of cusp forms by $L$-values, level aspect},
\newblock Journal of Number Theory 142(6), 305-321, 2014.

\bibitem[PS1979]{PS1979}
I. Piatetski-Shapiro, {\it Multiplicity one theorems},
    \newblock Automorphic Forms, Representations and $L$-Functions,
    \newblock Proceedings of the Symposium on Pure Mathematics, XXXIII, American Mathematical Society, 1979.

\bibitem[RS1996]{RS1996}
Z. Rudnick and P. Sarnak,
{\it Zeros of principal $L$ -functions and random matrix theory},
Duke Mathematical Journal 81(2), 269-322, 1996.

\bibitem[Su2014]{Su2014}
\newblock Q. Sun,
\newblock {\it On effective determination of symmetric-square lifts, level aspect},
\newblock International Journal of Number Theory 12(7), 976-990, 2014.


\bibitem[Zh2011]{Zh2011}
\newblock Y. Zhang,
\emph{Determining modular forms of general level by central values of convolution L-functions},
ACTA ARITHMETICA 150(1), 93-103, 2011.

\end{thebibliography}
\end{document}